\documentclass[final]{dmtcs-episciences}

\usepackage[utf8]{inputenc}
\usepackage{subfigure}
\usepackage{url}
\usepackage{amsmath, amssymb, amsthm}
\usepackage{makecell}

\newtheorem{theorem}{Theorem}
\newtheorem{conjecture}[theorem]{Conjecture}
\newtheorem{observation}[theorem]{Observation}
\newtheorem{corollary}[theorem]{Corollary}

\newtheorem{lemma}[theorem]{Lemma}

\newtheorem{definition}{Definition}

\newtheorem{notation}{Notation}

\theoremstyle{remark}
\newtheorem*{remark}{Remark}

\def\C{\mathcal{C}}

\def\bth{\begin{theorem}}
\def\eth{\end{theorem}}
\def\bc{\begin{corollary}}
\def\ec{\end{corollary}}
\def\bcj{\begin{conjecture}}
\def\ecj{\end{conjecture}}

\newcommand{\bizveg}{{\hfill $\Box$}}
\newcommand{\etal}{et\ al.\ }

\author[Lajos Gy\H{o}rffy et al.]
{Lajos Gy\H{o}rffy\affiliationmark{1,2}\thanks{\href{mailto:lgyorffy@math.u-szeged.hu}{lgyorffy@math.u-szeged.hu}, \url{https://orcid.org/0000-0002-4606-5390}}
  \and Andr\'as London\affiliationmark{3}\thanks{\href{mailto:london@inf.u-szeged.hu}{london@inf.u-szeged.hu}, \url{https://orcid.org/0000-0003-1957-5368}}
  \and G\'abor V. Nagy\affiliationmark{1,2}\thanks{\href{mailto:ngaba@math.u-szeged.hu}{ngaba@math.u-szeged.hu}, \url{https://orcid.org/0000-0002-3085-9620}}
  \and Andr\'as Pluh\'ar\affiliationmark{3}\thanks{\href{mailto:pluhar@inf.u-szeged.hu}{pluhar@inf.u-szeged.hu}, \url{https://orcid.org/0000-0001-6576-4202}\\ The research leading to these results has received funding from the national project TKP2021-NVA-09.
 Project no.\ TKP2021-NVA-09 has been implemented with the support provided by the Ministry of Culture and Innovation of Hungary from the National Research, Development
 and Innovation Fund, financed under the TKP2021-NVA funding scheme.}}
  
\title[Partitions of $K_n$ into Special Bipartite Graphs]{Partitions of $K_n$ into Special Bipartite Graphs}

\affiliation{
  University of Szeged, Bolyai Institute, Szeged, Hungary\\
  John von Neumann University, Kecskem\'et, Hungary\\
  University of Szeged, Institute of Informatics, Szeged, Hungary}
  
\keywords{Graph partition, Graham-Pollak, Ferrers graphs, nestedness, forbidden induced subgraphs}

\begin{document}
\publicationdata{vol. 27:3}{2025}{24}{10.46298/dmtcs.15361}{2025-03-12; 2025-03-12; 2025-10-21}{2025-10-27}
\maketitle
\begin{abstract}
We study the problem of partitioning the edge set of the complete graph into bipartite subgraphs under certain constraints defined by forbidden subgraphs. These constraints lead to both classical problems, such as partitioning into independent matchings or complete bipartite subgraphs, and novel variants motivated by structural restrictions. Our theoretical framework is inspired by clustering problems in real-world transaction graphs, which can be formulated naturally as edge partitioning problems under bipartite graph constraints.
	
The main result of this paper is the proof of the bounds for $\chi'_{2K_2}(n)$, which corresponds to the minimum number of induced $2K_2$-free bipartite subgraphs needed to partition the edges of $K_n$. In addition to this central result, we also present several similar bounds for other forbidden subgraphs on three or four vertices. 
Some are included primarily for the sake of completeness, to demonstrate the broad applicability of our approach, and some lead to other novel or well-known graph theoretical problems.
\end{abstract}

\section{Introduction and Results}
\subsection{Graph theoretical motivation}
A classical result of Graham and Pollak \cite{GP} shows that at least $n-1$ complete bipartite graphs are needed to partition the edge set of the complete graph $K_n$; it is easy to see that $n-1$ complete bipartite graphs are enough.
A natural generalization has led to the \emph{biclique partitioning problem}, where the goal is to partition the edges of a graph $G$ with the minimum number of bicliques. More generally, given a \emph{host} graph $G$, partition its edges into subgraphs belonging to a given set, called the template class. Elements of the template class are called template graphs, or simply \emph{templates} \cite{KU}. In the \emph{covering} version of the problem, every edge of the host graph must belong to at least one (but not necessarily exactly one, as in the case of partitioning) template. It is easy to see that $\lceil \log_2 n \rceil$ bicliques are sufficient to cover $K_n$  (cf.\ the aforementioned result of Graham and Pollak), showing that the gap between the partitioning and covering numbers for the same host graph can be significant. For a detailed overview, see \cite{Schwartz}.

In this study, we consider the graph partitioning problem under the constraint that the template class is defined by forbidden subgraphs. Specifically, we seek to partition the edge set of a given graph $G$ into subgraphs that avoid certain forbidden configurations. We investigate how different sets of forbidden subgraphs impact the partitioning structure and the minimum number of subgraphs required to partition the edges of $G$.

\subsection{Real-world motivation}
Our theoretical framework is primarily motivated by its applications in clustering real-world transaction graphs, where such constraints naturally emerge. Graph clustering and \emph{community detection} play a crucial role in graph-based data mining and in the development of models on graphs; see for example \cite{BCsP, NG, PDFV, Sch}. Motivated by applications such as social networks, these clustering methods aim to identify dense subgraphs and partition the graph into meaningful communities. While these methods often perform well in practice, defining a theoretically perfect clustering remains a challenging and unsolved problem, as discussed by Kleinberg \cite{Kleinberg}. 

Unlike clustering the nodes, edge partitioning, or edge-based clustering, allows each node to participate in multiple clusters, reflecting the fact that a single entity can belong to different groups depending on the context. This is particularly realistic in economic networks, where, for example, a company might interact with different industries through distinct transactions, such as being a supplier in one cluster and a customer in another. Generally, in the case of so-called {\em technological graphs} edge density based node clustering methods are not satisfactory. The background and examples are explained in \cite{Jun, LMP, Uzzi} and motivated the development of completely different approaches to clustering. Here we only recall some facts that are essential for understanding. Certain social networks, and especially technological or transaction networks, tend to have fewer triangles and often exhibit tree-like structures \cite{ASM}, making disjoint, dense clusters inherently inadequate. In addition, certain bipartite networks, such as pollination networks (of plant and pollinating animal species) or trade networks (of countries and imported/exported goods), often display special structures like \emph{nestedness} \cite{Bastolla, Uzzi, Wright}. In these networks, nodes can be ordered so that the neighborhood of each lower-ranked node contains that of any higher-ranked node. Ecological networks commonly exhibit this nested pattern, where specialists (lower-ranked) interacting with generalists (higher-ranked), which in turn interact with both specialists and other generalists. 

\subsection{Previous work}
In a former model, described in \cite{LMP}, the authors consider special \emph{good (vertex) colorings} of a graph, where the edge structure between any two color classes is restricted. A particularly important case is when the bipartite subgraphs between color classes are induced $2K_2$-free. This property is a structural concept also referred to as perfect nestedness in ecological networks.
Regarding computational complexity, some of the special colorings with a minimum number of colors can be found in polynomial time, however, deciding cases like the $2K_2$-free bipartite graphs is generally NP-complete.

A more application-driven approach is explored by Gera \etal \cite{GLP, GL}, where the goal is to find a covering of a graph such that the elements of the covering are not necessarily disjoint $2K_2$-free bipartite graphs. In this paper, we propose a model that lies between the purely theoretical and the highly application-oriented approaches.

\subsection{Definitions and notations}
After the brief overview, we turn to the discussion of the edge partitioning results of this paper. We begin with formal definitions and notations. We assume familiarity with concepts such as \emph{graphs, subgraphs, induced subgraphs, bipartite graphs, biclique, coloring, matching}, etc., which can be found, for example, in \cite{Diestel}. It is straightforward to check that induced $2K_2$-free bipartite graphs are precisely the so-called \emph{Ferrers graphs}.

\begin{definition}
A bipartite graph with color classes $A$ and $B$ is called a \emph{Ferrers graph} if there exists an ordering $a_1,\dots,a_n$ of the vertices in $A$ such that $N(a_1)\subseteq N(a_2)\subseteq\dots\subseteq N(a_n)$ holds, where $N(a_i)$ denotes the neighborhood of $a_i$. 
(This condition implies the same ``nestedness'' for the neighborhoods of vertices of $B$ as well.)
\end{definition}
Ferrers graphs are named for the property that the $1$'s in their bipartite adjacency matrix form a Ferrers diagram (of an integer partition) when the rows and columns are arranged in ``degree-decreasing'' order.
\begin{definition}
Fix a small graph $H$. An \emph{$H$-avoiding bipartite partition} of a graph $G$ is a set of bipartite graphs $G_1, \dots, G_k$ such that $E(G)=\cup_{i=1}^k E(G_i)$, $E(G_i) \cap E(G_j)=\emptyset$ if $i \not = j$ and $G_i$ does not contain $H$ as an induced subgraph for all $i$.
(The graphs $G_1,\dots,G_k$ are also called \emph{template graphs}.)
\end{definition}

\begin{notation}
If we have more than one graph $H$ (i.e. $H_1, H_2,\dots$), then we avoid a set of small graphs  $\mathcal{H}=\{H_1,H_2,\dots\}$ among the graphs $G_i$ and we write $\mathcal{H}$ instead of $H$.
\end{notation}

Of course, we want to have as few graphs as possible in an $H$-avoiding bipartite partition. 

\begin{notation}
For a fixed graph $H$ and a graph $G$, we denote by $\chi'_H(G)$ the smallest integer $k$ for which there exists an $H$-avoiding bipartite partition of $G$ consisting of $k$ graphs. We also use the notation $\chi'_H(n):=\chi'_H(K_n)$. For a set $\mathcal{H}$ of forbidden subgraphs, the notations $\chi'_{\mathcal{H}}(G)$ and $\chi'_{\mathcal{H}}(n)$ are defined analogously.
\end{notation}

\begin{observation}\label{obs_mono}
For any fixed set $\mathcal{H}$ of forbidden subgraphs, the sequence $\left(\chi'_{\mathcal{H}}(n)\right)_{n=1}^\infty$ is monotonically increasing, that is, $\chi'_{\mathcal{H}}(n)\le\chi'_{\mathcal{H}}(n+1)$ for all $n$.
\end{observation}
\begin{proof}
Consider an (optimal) $\mathcal{H}$-avoiding bipartite partition $G_1,\dots,G_t$ of $K_{n+1}$ where $t=\chi'_{\mathcal{H}}(n+1)$. Let $S$ be an $n$-element subset of $V(K_{n+1})$.
Then the induced subgraphs $G_1|_S,\dots,G_t|_S$ clearly form an $\mathcal{H}$-avoiding bipartite partition of $K_{n+1}|S\simeq K_n$ (empty graphs may appear here), proving that $\chi'_{\mathcal{H}}(n)\le t=\chi'_{\mathcal{H}}(n+1)$.
\end{proof}

We study mainly the case $G=K_n$. It motivates more general problems and highlights the difficulty of some examples even in this case. 

In Figure \ref{kiz} we list the small excluded subgraphs we will consider. We use $K_n, P_n, C_n$ and $S_n$ for the complete graph, path, cycle and star on $n$ vertices, respectively. The disjoint union of two graphs $H_1$ and $H_2$ with disjoint vertex sets is denoted by $H_1 + H_2$.

\begin{figure}[htbp]
	\centering
	\includegraphics[scale=0.41]{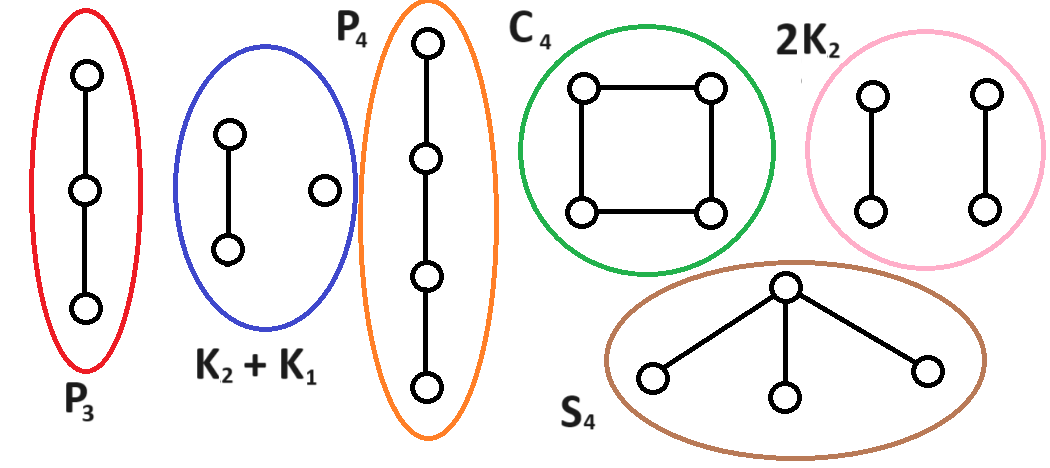}
	\caption {Small excluded subgraphs}
	\label{kiz}
\end{figure}

\begin{table}[htbp]
	\small
	\centering
	\caption{Results on $\chi'_{\mathcal{H}}(n)$}
	\label{tab:chi_h_summary}
	\begin{tabular}{|c|c|c|c|c|}
		\hline
		 & \textbf{ $H$ or $\mathcal{H}$} & \textbf{Lower/Upper Bound} & \textbf{Comment} & \textbf{Template Graphs} \\
		\hline
		\hline
		2.1 & $\emptyset$ & $\lceil\log_2 n\rceil$ & Folklore, Observation \ref{obs_empty}. & Any bipartite graph \\
		\hline
		\hline
		2.2 & $P_3$ & $n-1$ (even), $n$ (odd) & Chromatic index, Obs. 4. & Matchings \\
		\hline
		2.3 & $K_2+K_1$ & $n-1$ & Graham--Pollak, Theorem 5. & Bicliques \\
		\hline
		2.4 & $K_2+K_1, P_3$ & $\binom{n}{2}$ & Trivial, Observation 6. & Single edges \\
		\hline
		\hline
		2.5 & $P_4$ & $\lceil\log_2 n\rceil$ & Folklore, Observation 7. & Union of bicliques \\
		\hline
		2.6 & $C_4$ & $a\sqrt{n}$, $b\sqrt{n}\log n$ & \makecell[c]{Zarankiewicz (lower bound) \\ Projective planes (upper b.)\\ Theorem 8.} & $C_4$-free bipartite graphs \\
		\hline
		2.7 & $2K_2$ & \makecell[c]{$\lfloor\log_2 n\rfloor + \frac{1}{4} \sqrt{\lfloor\log_2 n\rfloor}-1$ \\ $2\left\lceil\sqrt{n}\right\rceil-2$} & \textbf{Main Result, Theorem \ref{NVGtetel}.} & Ferrers graphs \\
		\hline
		2.8 & $S_4$ & $\lceil(n-1)/2\rceil$ & Hamiltonian cycles, Obs. 10. & $\Delta(G)\le 2$ \\
		\hline
		\hline
		2.9 & $2K_2, C_4$ & $\left\lceil n/2 \right\rceil$ & \makecell[c]{Double star decomposition\\ Theorem \ref{ds}.} & Double stars \\
		\hline
		2.10 & $2K_2, C_4, P_4$ & $n-1$ & Graham--Pollak again & Stars \\
		\hline
		2.11 & $P_4, C_4, S_4$ & $\frac{3}{4}n+o(n)$ & Cherry orchard, Theorem 12. & Union of $P_3$'s and $K_2$'s \\
		\hline
		2.12 & $P_4, C_4$ & $\lceil n/2\rceil+1$ & Star orchard, Theorem 13. & Union of stars \\
		\hline
		\hline
		2.13 & $P_4, 2K_2$ & $n-1$ & Graham--Pollak, Obs. 14. & Bicliques \\
		\hline
		2.14 & $S_4, 2K_2$ & $n(n-1)/8$ & $C_4$ decomposition, Obs. 15. & $C_4, P_4, P_3, K_2$ \\
		\hline
		2.15 & $S_4, 2K_2, P_4$ & $n(n-1)/8$ & $C_4$ decomposition, Obs. 16. & $C_4, P_3, K_2$ \\
		\hline
		2.16 & $C_4, S_4$ & $\lceil(n-1)/2\rceil$ & Hamiltonian cycles, Obs. 17. & $\Delta(G)\le 2$, $C_4$-free \\
		\hline
		2.17 & $P_4, S_4$ & $n/2$ & $C_4$ orchard, Obs. 18.  & Union of $C_4$'s, $P_3$'s, $K_2$'s \\
		\hline
		2.18 & $C_4, S_4, 2K_2$ & $n(n-1)/6$ & $P_4$ decomposition, Obs. 19. & $P_4, P_3, K_2$ \\
		\hline
		2.19 & $P_4, C_4, S_4, 2K_2$ &  $\left\lceil n(n-1)/4\right\rceil$ & $P_3$ decomposition, Obs. 20. & Cherries ($P_3$ or $K_2$) \\
		\hline
	\end{tabular}
\end{table}

We consider only bipartite partitions $G_1,\dots,G_k$ in which no template graph has isolated vertices. We also assume that none of the template graphs is empty (as we are interested in partitions with a minimal number of template graphs).

\subsection{Overview of the results}
While our main focus is the case where the forbidden graph is $H = 2K_2$ (the induced matching of two edges)--motivated by applications in clustering real-world transaction graphs--we also consider several other small excluded graphs $H$ or sets of graphs $\mathbf{\mathcal{H}}$. These are: on three vertices $P_3$ and $K_2 + K_1$ and on four vertices $P_4, C_4, 2K_2$ and $S_4$, for mathematical completeness, following the approach in \cite{KKTW}. The values of $\chi'_{\mathcal{H}}(n)$ range from logarithmic in $n$ (specifically $\lceil\log_2 n\rceil$) up to quadratic in $n$ (specifically $\binom{n}{2}$).

The main result of this paper is the proof of the bounds for $\chi'_{2K_2}(n)$, presented in Section~\ref{sec:proof}. For smaller results presented in Section~\ref{sec:small} we only provide proof sketches in some case (especially 2.13-2.19). The complexity of determining $\chi'_{\mathcal{H}}(n)$ varies significantly across the cases presented in Table~\ref{tab:chi_h_summary}; some are trivial or relate to well-known problems, while others remain quite challenging. Some other complexity results are presented in Section~\ref{sec:complexity}.

\section{Classification of small excluded subgraphs}\label{sec:small}

\subsection{Partitioning into bipartite graphs without restriction}
It is a folklore fact that $K_n$ cannot be partitioned into less than $\log_2 n$ (arbitrary) bipartite graphs, and $\lceil\log_2 n\rceil$ bipartite graphs are enough. (We also sketch the proof in Theorem~\ref{NVGtetel}.)

\begin{observation}\label{obs_empty}
If there are no excluded graphs, then $\chi'_{\emptyset}(n)=\lceil\log_2 n\rceil$.
\end{observation}

\begin{observation}\label{obs_trivi}
$\chi'_{\mathcal H}(n)\ge\lceil\log_2 n\rceil$, for any set $\mathcal H$ of forbidden subgraphs.
\end{observation}

\subsection{$P_3$-free bipartite graphs}
If $H=P_3$, then the maximum degree in each template graph must be (at most) $1$, thus $G_i$ is a {\bf matching} for all $1 \le i \le k$. Therefore, the number of graphs in a minimal $P_3$-avoiding bipartite partition of $G$ is just the edge chromatic number of $G$,
which is equal to $\Delta(G)$ or $\Delta(G)+1$ by Vizing's theorem (where $\Delta(G)$ denotes the maximum degree of $G$). The edge chromatic number of the complete graph $G=K_n$ is well known, which leads to the following observation.


\begin{observation}
$\chi'_{P_3}(n)=n-1$, if $n$ is even; and $\chi'_{P_3}(n)=n$, if $n$ is odd.
\end{observation}

\begin{remark}
The notation $\chi'_{\mathcal H}(G)$ comes from the fact that $\chi'_{P_3}(G)=\chi'(G)$ for all graphs $G$, where $\chi'(G)$ is the chromatic index of the graph $G$.
\end{remark}

\subsection{$(K_2+K_1)$-free bipartite graphs}
If $H=K_2+K_1$ then there cannot be a single point independent of any other edge. Therefore, $G_i$ is a {\bf complete bipartite graph} for all $1 \le i \le k$ and the number of graphs in a minimal $H$-avoiding bipartite partition of $G$ gives the well-known Graham-Pollak theorem. 

\begin{theorem}[Graham--Pollak \cite{GP}]
There are at least $n-1$ subgraphs in a complete bipartite partition of $K_n$ (and $n-1$ subgraphs are enough), that is, $\chi'_{K_2+K_1}(n)=n-1$.
\end{theorem}

Most of the proofs of this theorem use linear algebraic methods. No purely combinatorial proof is known (except a sophisticated one by Vishwanathan \cite{vish}, which can be seen as a translation of the linear algebraic method). Even for weaker lower bounds, e.g. $\sqrt n$, we do not know any nice combinatorial proof.

\subsection {$\{K_2+K_1, P_3\}$-free bipartite graphs}
If $\mathcal{H}=\{K_2+K_1,P_3\}$, then it is easy to see that the template graphs $G_i$ are {\bf single edges} for all $1 \le i \le k$. That is, the number of graphs in a minimal $H$-avoiding bipartite partition of $G$ is just the number of edges $n \choose 2$. 

\begin{observation}
If $\mathcal{H}=\{K_2+K_1,P_3\}$, then $\chi'_{\{K_2+K_1,P_3\}}(n)={n \choose 2}$.
\end{observation}

We have listed all possible cases among forbidden graphs with three vertices. Now we consider all subsets $\mathcal{H}$ of the set of $4$-vertex excluded subgraphs shown in Figure~\ref{kiz}.

\subsection{$P_4$-free bipartite graphs}
If $H=P_4$, then it is easy to check that $G_i$ is a {\bf vertex disjoint collection of complete bipartite graphs} for all $1 \le i \le k$. If we divide the complete graph into two (almost) equal sets and draw all edges between the two color classes in $G_1$, then there are $n^2/4$ edges and in $G_2$ we get two pieces of complete graphs with $n/2$ vertices where we can continue the division. (We do not require connectedness, so we can include more bipartite graphs in a $G_i$.)
It shows that if $n$ is a power of 2, then we can always divide the remaining graph into two parts, which yields $\log_2 n$ template graphs; and hence $\chi'_{P_4}(n)=\log_2 n$ by Observation~\ref{obs_trivi}.
In any other case, we can bound $\chi'_{P_4}(n)$ by $\chi'_{P_4}(2^{\lceil\log_2 n\rceil})$ using the monotonicity of $\chi'_{P_4}(n)$.

\begin{observation}
If $H=P_4$, then $\chi'_{P_4}(n)=\lceil\log_2 n\rceil$.
\end{observation}

If the graphs of the partition are required to be connected, we come to the assumption of the Graham-Pollak theorem.

\subsection{$C_4$-free bipartite graphs}
If $H=C_4$, then $G_i$ is a $C_4$-free bipartite graph for all $1 \le i \le k$. Since $|E(G_i)| \le c_1n^{3/2}$ (see Zarankiewicz \cite{zar}), we immediately get a lower bound $c_2n^{1/2} \le \chi'_{C_4}(n)$.

However, the upper bound is still open. If we throw the point-line graph of large independent projective planes onto the complete graph $K_n$ $k$ times, then the probability that an edge is uncovered is about $(1-c/\sqrt{n})^k$. If some edges are over covered (covered more than once), then we simply delete the surplus from the graphs after the first cover. By setting $k=(3/c)\cdot\sqrt{n}\log n$ the Boole inequality readily gives the following, probably not sharp, upper bound: $\chi'_{C_4}(n) < \log n \cdot \sqrt n$, since

$${\mathrm{Pr}}(\mathrm{edge}\ e\ \mathrm{is\ not\ covered}) \leq (1-c/\sqrt{n})^k \leq
e^{-kc/\sqrt{n}}=n^{-3},$$
and so
$${\mathrm{Pr}}({\mathrm{some\ edge\ from}}\ E(K_n)\ \mathrm{is\ not\ covered}) \leq
\sum_1^{\binom{n}{2}} n^{-3} < n^{-1} <1,$$
provided $n$ is big enough.

\begin{theorem}
For some positive constants $a$ and $b$,
$$a\sqrt{n}\le\chi'_{C_4}(n)\le b\sqrt{n}\log n,$$
if $n$ is large enough.
\end{theorem}

We think that the next subsection is the most intriguing part of the paper, which illustrates that these types of questions can be difficult. The lower and upper bounds are very far apart, and the fact that there is no easy combinatorial proof of the Graham-Pollak theorem also shows the difficulty of the problem. However, it gives us a remarkable open problem.

\subsection{$2K_2$-free bipartite graphs}
If $H=2K_2$, then each $G_i$ is a {\bf Ferrers graph}. The next result gives a non-trivial lower and upper bound on $\chi'_{2K_2}(n)$.
\begin{theorem}\label{NVGtetel}
$\lfloor\log_2 n\rfloor + \frac{1}{4} \sqrt{\lfloor\log_2 n\rfloor}-1 < \chi'_{2K_2}(n)\leq 2\left\lceil\sqrt{n}\right\rceil-2$.
\end{theorem}

Theorem \ref{NVGtetel} is proved in the next section.

\subsection{$S_4$-free bipartite graphs}
If $H=S_4$, then the maximum degree of each $G_i$ is at most 2. Therefore, we need at least $\lceil(n-1)/2\rceil$ template graphs. On the other hand, Walecki \cite{wal} showed that $K_n$ can be decomposed into $(n-1)/2$ Hamiltonian cycles for odd $n$, yielding the upper bound $(n-1)/2$. Therefore the exact result is $\chi'_{S_4}(n)=\lceil(n-1)/2\rceil$. (For even $n$, the upper bound follows from the inequality $\chi'_{S_4}(n)\le\chi'_{S_4}(n+1)$.)

\begin{observation}
If $H=S_4$, then $\chi'_{S_4}(n)=\lceil(n-1)/2\rceil$.
\end{observation}

\begin{remark}
The existence of a partition into Hamilton cycles in arbitrary graphs is an NP-complete problem. Even in 4-regular graphs.
\end{remark}

\subsection {$\{2K_2, C_4\}$-free bipartite graphs}
If we avoid $2K_2$ and $C_4$ together, then the subgraphs $G_i$ are connected and do not contain any induced cycle, because $C_6, C_8, \dots$ contain induced $2K_2$. Hence, the template graphs are trees with $n-1$ edges and we get the lower bound $\left\lceil n/2 \right\rceil$. Moreover, the trees are $P_4$-free, hence they are {\bf double stars}.

On the other hand, if we use double stars for partitioning $K_n$, this is enough, as we can see in Figure \ref{doub} for $n=6$ and a similar construction works for any even $n$.
(The exact upper bound for odd $n$ follows from the monotonicity of $\chi'_{\{2K_2, C_4\}}(n)$.)

\begin{theorem} \label{ds}
	If $\mathcal{H}=\{2K_2, C_4\}$, then $\chi'_{\{2K_2, C_4\}}(n)= \left\lceil n/2 \right\rceil$.
\end{theorem}

\begin{figure}[htbp]
	\centering
	\includegraphics[scale=0.41]{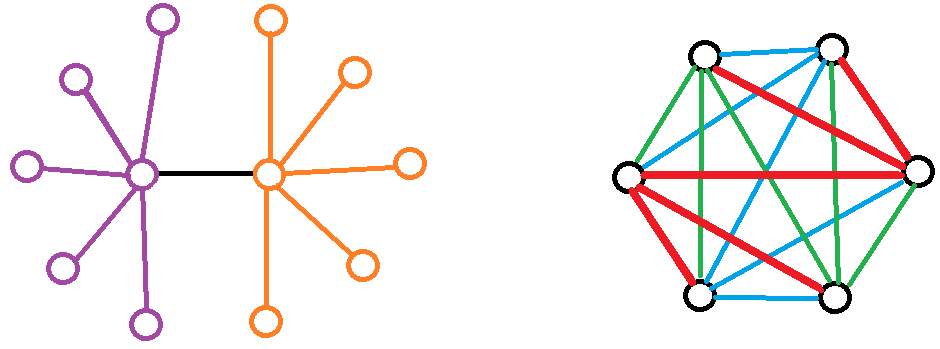}
	\caption {Double star partition of $K_6$.}
	\label{doub}
\end{figure}

\subsection {$\{2K_2, C_4, P_4\}$-free bipartite graphs}
If we also omit $P_4$, then double stars are not allowed, just simple {\bf stars}. What remains is a star decomposition which comes from the well-known solution of the Graham-Pollak theorem, and gives $n-1$ as the best possible result.

\begin{figure}[!t]
	\centering
	\includegraphics[scale=0.41]{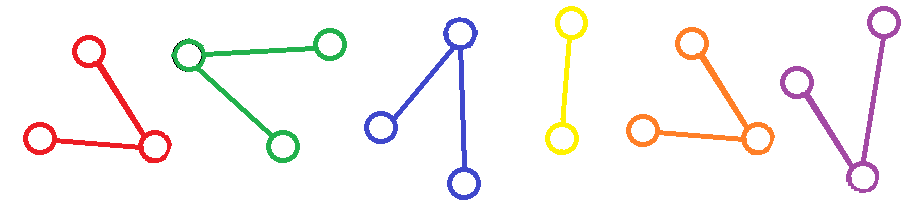}
	\caption {Cherry Orchard.}
	\label{cher}
\end{figure}

\subsection {$\{C_4, P_4, S_4\}$-free bipartite graphs}
If $P_4, S_4, C_4$ are the excluded subgraphs and the $G_i$ graphs are bipartite, then the maximum degree is two and there is no cycle of any length in $G_i$. Therefore, each $G_i$ consists of single edge and $P_3$ components, that is, the $G_i$ graphs are a set of {\bf disjoint cherries} (we will call them {\bf cherry orchards}), for all $i$. One can see a single cherry orchard in Figure \ref{cher}. In a cherry orchard there are at most $\frac{2}{3}n$ edges. The lower bound to $\chi'_{\mathcal{H}}(n)$ comes from simply counting the edges: ${n \choose 2}/\frac{2n}{3} \approx \frac{3}{4}n$.

On other side we can reach the $\frac{3}{4}n+o(n)$ upper bound in the following way. We divide the vertices of $K_n$ into three (almost) equal parts, let us call them $A, B$ and $C$, all containing $n/3$ vertices. The edges between each pair of sets can be divided into $n/3$ perfect matchings.
We divide these $3 \cdot \left(\frac{n}{3}\right)^2$ edges into cherry orchards. First, we consider three cherry orchards, each with $\frac{2n}{3}$ edges. The first is the union of a perfect matching from $(A,B)$ and $(A,C)$. The other two are the unions of perfect matchings from $(B,A)$ and $(B,C)$; and from $(C,A)$ and $(C,B)$, respectively.
Then we remove these three cherry orchards, and all degrees of vertices in the sets decreased by exactly four. 
We do the same in the next steps, where each step consists of three cherry orchards, until we run out of edges. With this procedure three cherry orchards decreases the degrees of vertices by four and we started with $\frac{2n}{3}$ degrees at each vertex at the beginning. So with $\frac{2n/3}{4}\cdot 3=\frac{n}{2}$ cherry orchards we covered all the edges between the parts.

Then we apply the same procedure recursively on the sets $A, B$ and $C$. This can be done on each set separately. So, if $T(n)$ denotes the number of cherry orchards needed, then
$$T(n) \leq \frac{1}{2}n + T\left(\frac{n}{3}\right) = \frac{1}{2}n + \frac{1}{6}n + T\left(\frac{n}{9}\right)$$
$$= \left(\frac{1}{2} + \frac{1}{6} + \frac{1}{18} +\dots \right)n = \frac{n}{2}\left(1+\frac{1}{3}+\frac{1}{9}+\dots\right) = \frac{3}{4}n+o(n).$$

\begin{theorem}[Cherry Orchard]
	If $\mathcal{H}=\{P_4, C_4, S_4\}$, then
	$\chi'_{\mathcal{H}}(n)=\frac{3}{4}n+o(n).$
\end{theorem}

\subsection {$\{C_4, P_4\}$-free bipartite graphs}
If we do not require connectedness or maximum degree at most 2 (since we do not forbid $2K_2$ or $S_4$), then each template graph $G_i$ can consist of {\bf arbitrarily many star components} with arbitrary degrees. We call them {\bf star orchards}.
A trivial lower bound $\lceil\frac{n}{2}\rceil$ comes from the fact that the template graphs contain at most $n-1$ edges.

Assume first that $n$ is even. An upper bound can be derived from the former double star construction (see Figure~\ref{doub}):
If we delete the central edge of each double star, then we have two disjoint stars in each graph $G_i$ without any induced $P_4$. The remaining extra edges form a matching that can be considered as a set of disjoint stars and the last template graph in the partition. This gives the upper bound $\frac n2+1$.
In fact, this is the exact value of $\chi'_{\{P_4, C_4\}}(n)$ for even $n$:
observe that at most one (star orchard) template graph can have $n-1$ edges, i.e.\ at most one template graph can be a single star on $n$ vertices, because the center of one such star cannot be contained in any other template graph.
This means that the trivial lower bound $\frac{n}{2}$ cannot be sharp for even $n$, if $n\ge4$, because each template would have exactly $n-1$ edges in a star orchard partition of $K_n$ with $\frac{n}{2}$ template graphs.
So $\chi'_{\{P_4, C_4\}}(n)=\frac n2+1$ for even $n\ge4$.

It turns out that $\chi'_{\{P_4, C_4\}}(n)=\lceil\frac n2\rceil+1$ holds for odd $n\ge5$ as well. This case can be reduced to the case of even $n$.

If $n$ is odd, then a minimal star orchard partition of $K_{n-1}$ can be extended to a star orchard partition of $K_n$ with $\frac{n-1}2+2=\lceil\frac n2\rceil+1$ template graphs by adding an $n$-vertex single star template.
However, $\lceil\frac{n}{2}\rceil$ templates are not enough. This is because $K_n$ cannot be partitioned into $\lceil\frac n2\rceil$ star orchards such that each template graph has at most $n-2$ edges,
as then the total number of edges in the templates would be at most
$$\left\lceil\frac n2\right\rceil\cdot(n-2)=\frac {n+1}2\cdot(n-2)<\frac{n(n-1)}2=E(K_n).$$
Finally, for odd $n\ge5$, $K_n$ cannot be partitioned into $\lceil\frac n2\rceil$ star orchards such that one of them is a single $n$-vertex star,
because the removal of the $n$-vertex star template would yield a star orchard partition of $K_{n-1}$ with $\lceil\frac n2\rceil-1=\frac{n-1}2$ templates,
which is impossible as seen in the previous paragraph. So we obtained the following:
\begin{theorem}
	If $\mathcal{H}=\{P_4, C_4\}$, then $\chi'_{\mathcal{H}}(n) = \lceil\frac{n}{2}\rceil+1$, for $n\ge4$.
\end{theorem}


The discussion of further examples will be less detailed. For example, we only include results where $n$ meets some ``natural'' divisibility conditions. 

\subsection {$\{P_4, 2K_2\}$-free bipartite graphs} Here the template graphs can be bicliques only. Thus, we get back the Graham-Pollak theorem.
	\begin{observation}
		If $\mathcal{H}=\{P_4, 2K_2\}$, then 
		$\chi'_{\mathcal{H}}(n)=n-1.$
	\end{observation}

\subsection {$\{2K_2, S_4\}$-free bipartite graphs} In this case, the templates can be only $C_4, P_4, P_3$ or $K_2$. The trivial lower bound can be attained
by the $C_4$-partition of $K_n$ (with divisibility conditions on $n$), see~\cite{SE}.
	\begin{observation}
		If $\mathcal{H}=\{S_4, 2K_2\}$, then
		$\chi'_{\mathcal{H}}(n)=\frac{n(n-1)}{8}$, when $n\equiv1\mod 8$.
	\end{observation}

\subsection {$\{2K_2, S_4, P_4\}$-free bipartite graphs} Now the templates are only $C_4, P_3$ or $K_2$. 
An optimal partition is given by the same construction as in the previous case (for the same divisibility condition).

\begin{observation}
	If $\mathcal{H}=\{S_4, 2K_2, P_4\}$, then
	$\chi'_{\mathcal{H}}(n)=\frac{n(n-1)}{8}$, when $n\equiv1\mod 8$.
\end{observation}

\subsection {$\{C_4, S_4\}$-free bipartite graphs}
In this case the maximum degree is at most $2$ in each template (and there are no $C_4$ components). The Hamilton cycle decomposition (Walecki construction) of $K_n$ is optimal here, as in case of $S_4$-free partitioning.
\begin{observation}
	If $\mathcal{H}=\{C_4, S_4\}$, then
	$\chi'_{\mathcal{H}}(n)=\left\lceil\frac{n-1}{2}\right\rceil$, when $n>4$.
\end{observation}

\subsection {$\{P_4, S_4\}$-free bipartite graphs} Here, the templates are disjoint $C_4$ components (and their induced subgraphs), call them $C_4$ orchards.
Danzinger et al. \cite{Danzig} showed that if $n$ is divisible by $4$, then $K_n$ can be partitioned into one perfect matching and $n/2-1$ graphs $(n/4)C_4$, where $(n/4)C_4$ denotes the vertex-disjoint union of $n/4$ cycles $C_4$;
see Theorem 2.2 in \cite{Danzig}.
This is a minimal $\{P_4, S_4\}$-free bipartite partition of $K_n$ with $n/2$ templates.
	\begin{observation}
		If $\mathcal{H}=\{P_4, S_4\}$, then
		$\chi'_{\mathcal{H}}(n)=\frac{n}{2}$, when $n$ is divisible by $4$.
	\end{observation}

\subsection {$\{C_4, S_4, 2K_2\}$-free bipartite graphs} The templates are only the paths $P_4$, $P_3$ or $K_2$. In this case an optimal partition can be obtained again from the Hamilton cycle decomposition of $K_n$, by cutting each Hamilton cycle into $n/3$ $P_4$'s (with divisibility conditions on $n$).
\begin{observation}
	If $\mathcal{H}=\{C_4, S_4, 2K_2\}$, then
	$\chi'_{\mathcal{H}}(n)=\frac{n(n-1)}{6}$, when $n\equiv3\mod 6$.
\end{observation}

\subsection {$\{C_4, P_4, S_4, 2K_2\}$-free bipartite graphs} In this case, the templates are {\bf single cherries}.
It is a known corollary of Tutte's theorem on perfect matchings that every connected graph with an even number of edges can be decomposed into $P_3$-s. Hence the trivial lower bound is sharp.
	\begin{observation}
		If $\mathcal{H}=\{P_4, C_4, S_4, 2K_2\}$, then
		$\chi'_{\mathcal{H}}(n)=\left\lceil\frac{n(n-1)}{4}\right\rceil.$
	\end{observation}

\section{The proof of Theorem~\ref{NVGtetel}.}\label{sec:proof}
\begin{proof}[of the upper bound]
We construct a ($2K_2$-free bipartite) partition of $K_n$ consisting of at most $2\left\lceil\sqrt{n}\right\rceil-2$ Ferrers graphs, for any $n\ge1$.

First, we assume that $n$ is a square number, that is, $\left\lceil\sqrt{n}\right\rceil=\sqrt{n}$ holds.
In this case we can also assume that the vertex set of $K_n$ is the set $\mathcal{V}:=\{1,\dots,\sqrt{n}\}\times\{1,\dots,\sqrt{n}\}$.
We divide the edge set of $K_n$ into two classes as follows. Let $e$ be an edge of $K_n$ between the vertices $(x,y)$ and $(x',y')$,
where $x\le x'$. We say that $e$ is a \emph{descent} if $x<x'$ and $y\ge y'$.
Otherwise we say that $e$ is an \emph{ascent}, i.e.\ $e$ is an ascent if $x<x'$ and $y<y'$, or $x=x'$ (and $y\ne y'$).
See also Figure~\ref{fig_inc_dec}. (If the vertices are regarded as points in the Cartesian plane and the edges are drawn as straight line segments,
then the sign of the slope of an edge $e$ determines whether $e$ is a descent or an ascent.) Note that by our definition, ``horizontal'' edges are descents and ``vertical'' edges are ascents.
By the \emph{left endpoint} of a descent edge $e$ we mean the (unique) endpoint of $e$ which has smaller first coordinate;
by the \emph{lower endpoint} of an ascent edge $f$ we mean the (unique) endpoint of $f$ which has smaller second coordinate.
\begin{figure}[htbp]%
\includegraphics[scale=0.8]{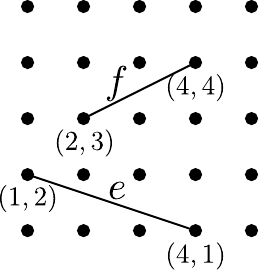}%
\centering
\caption {A descent edge $e$ and an ascent edge $f$ in $K_{25}$}\label{fig_inc_dec}%
\end{figure}%

After these preliminaries, we define $2\sqrt{n}-2$ Ferrers graphs, $G_1,G_2\dots,G_{\sqrt{n}-1},\allowbreak G'_1,G'_2,\dots,G'_{\sqrt{n}-1}$,
which gives a bipartite partition of $K_n$. For $i=1,\dots,\sqrt{n}-1$, let $G_i$ be the bipartite graph with color classes
$A_i=\{(x,y)\in\mathcal{V}: x=i\}$ and $B_i=\{(x,y)\in\mathcal{V}: x>i\}$ such that the vertices $a\in A_i$ and $b\in B_i$ are adjacent in $G_i$
if and only if $ab$ is a descent edge of $K_n$. For $j=1,\dots\sqrt{n}-1$, let $G'_j$ be the bipartite graph with color classes
$A'_j=\{(x,y)\in\mathcal{V}: y=j\}$ and $B'_j=\{(x,y)\in\mathcal{V}: y>j\}$ such that the vertices $a\in A'_j$ and $b\in B'_j$ are adjacent in $G'_j$
if and only if $ab$ is an ascent edge of $K_n$. 
\begin{figure}[htbp]%
\includegraphics[scale=0.8]{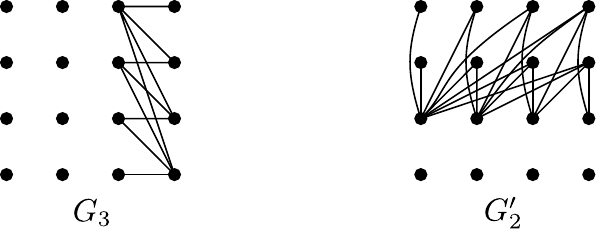}%
\centering
\caption {Illustration of $G_i$ and $G'_j$ (for $n=16$)}\label{fig_D3_I2}%
\end{figure}%

First we check that these bipartite graphs are indeed Ferrers graphs. In a graph $G_i$, the neighborhood $N(i,y_0)$ of the vertex $(i,y_0)\in A_i$
is the set $\{(x,y)\in\mathcal{V}:x>i,\,y\le y_0\}$; and so we actually have the nested property
\[N(i,1)\subseteq N(i,2)\subseteq N(i,3)\subseteq\dots\subseteq N(i,\sqrt n).\]
In a graph $G'_j$, the neighborhood $N(x_0,j)$ of the vertex $(x_0,j)\in A'_j$
is the set $\{(x,y)\in\mathcal{V}:x\ge x_0,\,y>j\}$; and so we have the nested property
\[N(1,j)\supseteq N(2,j)\supseteq N(3,j)\supseteq\dots\supseteq N(\sqrt n,j)\]
again. 
Now all that remains is to show that the graphs $G_1,\dots,G_{\sqrt{n}-1},G'_1,\dots,G'_{\sqrt{n}-1}$
partition $K_n$. Observe that for all $i=1,\dots,\sqrt{n}-1$, the graph $G_i$ contains all descent edges of $K_n$
whose left endpoint has the first coordinate $i$, and for all $j=1,\dots,\sqrt{n}-1$, the graph $G'_j$ contains all ascent edges of $K_n$
whose lower endpont has the second coordinate $j$; and all edges of $K_n$ are contained in exactly one of these $2\sqrt{n}-2$ edge classes.
Thus we proved that $\chi'_{2K_2}(n)\le 2\sqrt{n}-2$ for all square numbers $n$.

If $n$ is a non-square number, then consider the closest larger square number $n^*:=\left(\lceil\sqrt n\rceil\right)^2$ (for which the previous case applies),
and bound $\chi'_{2K_2}(n)$ using its monotonicity (see Observation~\ref{obs_mono}):
$$\chi'_{2K_2}(n)\le\chi'_{2K_2}(n^*)\le2\sqrt{n^*}-2=2\left\lceil\sqrt{n}\right\rceil-2.$$
So, the upper bound is proved for all positive integers $n$.
\end{proof}

Before proving the lower bound, we present two required lemmas.
\begin{lemma}\label{NVGlemma1} If a Ferrers graph $F$ contains a set of $m$ independent edges, then $F$ has at least $m(m+1)/2$ edges.
\end{lemma}
\begin{proof} Let $A$ and $B$ be the two color classes of $F$. Now let $u_1v_1,u_2v_2,\dots,u_mv_m$
be $m$ independent edges of $F$, where $u_1,\dots,u_m\in A$, $v_1,\dots,v_m\in B$. Since $F$ is a Ferrers graph,
we can assume without loss of generality that $$N(u_1)\subseteq N(u_2)\subseteq\dots\subseteq N(u_m).$$
This chain implies that $u_iv_j$ is also an edge of $F$ for all $1\le j\le i\le m$, cf.\ Figure~\ref{fig_NVGlemma1}.
(The edge $u_iv_i$ guarantees that $v_i\in N(u_i)$, for $i=1,\dots,m$. Since $v_1\in N(u_1)\subseteq N(u_2)$, thus $u_2$ is also adjacent to $v_1$.
Then $v_1,v_2\in N(u_2)\subseteq N(u_3)$ implies that $u_3$ is adjacent to $v_1$ and $v_2$; and so on.)
We have just found $1+2+\dots+m=m(m+1)/2$ distinct edges in $F$.
\begin{figure}[htbp]%
\includegraphics[scale=0.35]{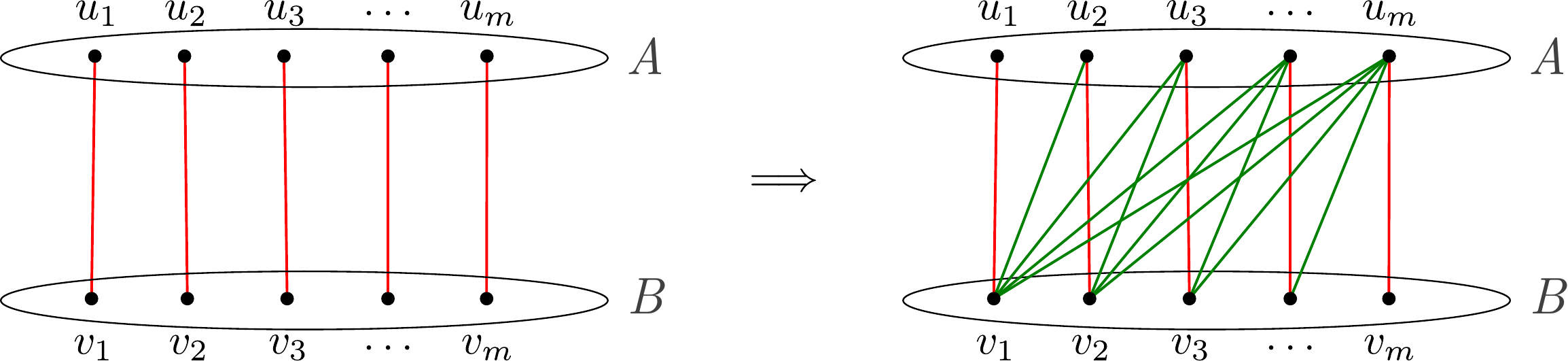}%
\centering
\caption {Illustration of the proof of Lemma~\ref{NVGlemma1}}\label{fig_NVGlemma1}%
\end{figure}%
\end{proof}
\def\CL{\text{CL}}
\begin{lemma}\label{NVGlemma2} For fixed integers $k\ge1$ and $d\ge0$,
assume that $K_{2^k}$, the complete graph on $2^k$ vertices, is partitioned into the bipartite graphs $G_1,G_2,\dots,G_{k+d}$.
Then there exist (at least) $k$ graphs in $\{G_1,\dots,G_{k+d}\}$ such that each contains a set of at least $\frac{2^{k-2}}{d+1}$ independent edges.
\end{lemma}
\begin{proof} 
Let $A_i$ and $B_i$ be the two color classes of $G_i$, for $i=1,\dots,k+d$. We can assume that $A_i\cup B_i = V(K_{2^k})$
by adding isolated vertices to the graph $G_i$ if necessary. To each vertex $v$ of $K_{2^k}$
we assign a vector $\CL(v)\in\{0,1\}^{k+d}$, whose $i^{\text{th}}$ coordinate is defined as $0$, if $v\in A_i$, and it is defined as $1$, 
if $v\in B_i$, for $i=1,\dots,k+d$.
We call the vector $\CL(v)$ the \emph{class-vector} of $v$. We mention some properties of class-vectors:
\begin{enumerate}
\item[(i)] If $u\ne v$, then $\CL(u)\ne\CL(v)$.
\item[(ii)] For fixed vertices $u\ne v$, let $S\subseteq\{1,\dots,k+d\}$ be the set of coordinates where $\CL(u)$ and $\CL(v)$ differ,
and let $G_i$ be the graph of the partition containing the edge $uv$ (of $K_{2^k}$). Then the index $i$ must be an element of $S$.
\end{enumerate}
Both properties follow directly from the fact that if the edge $uv\in E(K_{2^k})$ is contained in the graph $G_i$ of the partition, then
$\CL(u)$ and $\CL(v)$ differ in coordinate~$i$, because $u$ and $v$ are in different color classes of $G_i$. (In fact, property (i) is a corollary of (ii).)\par
The rest of the proof relies only on the information encoded in $\mathcal{C}:=\{\CL(v):v\in V(K_{2^k})\}$,
the set of class-vectors. Recall that $\mathcal{C}\subseteq\{0,1\}^{k+d}$, and note that $|\mathcal{C}|=|V(K_{2^k})|=2^k$ by property (i) above. From now on, we identify the vertices of $K_{2^k}$ with their class-vectors, and so $V(K_{2^k})=\mathcal C$ with a slight abuse of notation.
We say that a graph $G$ is \emph{large} if it contains a set of at least $\frac{2^{k-2}}{d+1}$ independent edges; 
otherwise we say that $G$ is \emph{small}. We need to prove that there are at most $d$ small graphs
among $G_1,\dots,G_{k+d}$, i.e.\ this bipartite partition cannot have $d+1$ small graphs.
To this end, we prove that there exists at least one large graph among arbitrary $d+1$ graphs from $\{G_1,\dots,G_{k+d}\}$. Without loss of generality,
we show that at least one of the graphs $G_1,\dots,G_{d+1}$ is large. We group the class-vectors in $\C$ (vertices of $K_{2^k}$) by their last $k-1$ coordinates:
For any $\bold b\in\{0,1\}^{k-1}$, let $V_{\bold b}\subseteq \mathcal{C}$ be the set of those class-vectors (vertices) whose last $k-1$ coordinate is $\bold b$.
($V_{\bold b}$ might be the empty set.) The point is that for any fixed $\bold b\in\{0,1\}^{k-1}$, in $K_{2^k}$, every edge between the vertices of $V_{\bold b}$ must be
contained in one of the graphs $G_1,\dots,G_{d+1}$, by property (ii) of class-vectors. Since these edges form a complete graph $K_{\bold b}$ on vertex set $V_{\bold b}$,
we can trivially find a matching $M_{\bold b}$ of size $\left\lfloor\frac{|V_{\bold b}|}{2}\right\rfloor$ in $K_{\bold b}$. If we do this for all sets
$V_{\bold b}$ (for all $\bold b\in\{0,1\}^{k-1}$), and take the union of the resulting independent edge sets $M_{\bold b}$,
we yield an independent edge set $M$ of size
\begin{multline*}
|M|=\sum_{\bold b\in\{0,1\}^{k-1}}\left\lfloor \frac{|V_{\bold b}|}2\right\rfloor\ge
\sum_{\bold b\in\{0,1\}^{k-1}}\left(\frac{|V_{\bold b}|}2-\frac12\right)=\\
\frac12\left(\sum_{\bold b\in\{0,1\}^{k-1}}|V_{\bold b}|\right)-\frac12\cdot{2^{k-1}}=
\frac12|\mathcal{C}|-2^{k-2}=
\frac12\cdot{2^k}-2^{k-2}=2^{k-2}
\end{multline*}
in $K_{2^k}$. Since every edge of $M$ is contained in one of the graphs $G_1,\dots,G_{d+1}$, thus at least one of these $d+1$ graphs 
must contain at least $\frac{2^{k-2}}{d+1}$ independent edges (from $M$) by the pigeonhole principle, i.e.\ there is a large graph among $G_1,\dots,G_{d+1}$.
This is what we wanted to prove, so the proof is now complete.
\end{proof}

\begin{proof}[of the lower bound in Theorem~\ref{NVGtetel}] First we assume that $n=2^k$ for some positive integer $k$. In this case
we have to prove that $\chi'_{2K_2}(2^k)>k + \frac{1}{4}\sqrt{k} - 1$. It is obvious that $\chi'_{2K_2}(2^k)\ge k$ because, as mentioned in Observation~\ref{obs_trivi}, it is a known (folklore) fact
that $K_{2^k}$ cannot be partitioned into less than $k$ (arbitrary) bipartite graphs. (In the terminology of the proof of Lemma~\ref{NVGlemma2},
a bipartite partition containing less than $k$ graphs would result $2^k$ different class-vectors with less than $k$ coordinates, which is a contradiction.)
Now assume that the Ferrers graphs $F_1,\dots,F_{k+d}$ partition $K_{2^k}$ such that $k+d$ is minimal, i.e.\ $k+d=\chi'_{2K_2}(2^k)$. We already know that
$d\ge0$, and we want to prove that $d>\frac{1}{4}\sqrt{k} - 1$ also holds. We apply Lemma~\ref{NVGlemma2} first with $G_i:=F_i$, for $i=1,\dots,k+d$:
This lemma guarantees that there exist $k$ graphs $F_{i_1},\dots,F_{i_k}$ in our partition, each having at least $\frac{2^{k-2}}{d+1}$ independent edges.
Then, by Lemma~\ref{NVGlemma1}, each $F_{i_s}$ has at least $\frac12\cdot\frac{2^{k-2}}{d+1}\cdot\left(\frac{2^{k-2}}{d+1}+1\right)$ edges, for $s=1,\dots,k$. We have 
found altogether at least 
$\frac k2\cdot\frac{2^{k-2}}{d+1}\cdot\left(\frac{2^{k-2}}{d+1}+1\right)$
distinct edges in $K_{2^k}$, which is bounded by $\frac{2^k(2^k-1)}2$, the total number of edges of $K_{2^k}$. The inequality
$$
\frac{k}2{\left(\frac{2^{k-2}}{d+1}\right)\!}^2<\frac k2\cdot\frac{2^{k-2}}{d+1}\cdot\left(\frac{2^{k-2}}{d+1}+1\right)\le\frac{2^k(2^k-1)}{2}<\frac12\cdot4^k$$
can be transformed into
$$\frac{\sqrt k}{4}<d+1$$
by a simple calculation, which completes the proof for the case $n=2^k$. \par
If $n$ is not a power of $2$, let $k^*$ be the largest integer such that $2^{k^*}\le n$, i.e.\ set $k^*:=\lfloor\log_2 n\rfloor$.
Then the previous case applies for $n^*:=2^{k^*}$, so by the monotonicity of $\chi'_{2K_2}(n)$ (see Observation~\ref{obs_mono}) we have that
$$\chi'_{2K_2}(n)\ge\chi'_{2K_2}(2^{k^*})>k^* + \frac{1}{4}\sqrt{k^*} - 1=\lfloor\log_2 n\rfloor + \frac{1}{4}\sqrt{\lfloor\log_2 n\rfloor} - 1,$$
as stated.
\end{proof}

\section{Complexity issues}\label{sec:complexity}
So far, we have mainly examined partitions of $K_n$. Finally we address the problem of computing $\chi'_{\mathcal H}(G)$ for a fixed $\mathcal{H}$ and arbitrary graph $G$. Some cases are the following.

\begin{observation} \label{edges} If
	${\mathcal H}=\{P_3, K_2+K_1\}$ then $\chi'_{\mathcal H}(G)=e(G)$, where $e(G)$ denotes the number of edges of $G$. 
\end{observation}

\begin{observation} \label{chromatic_index}
	If ${\mathcal H}=\{P_3\}$ then the language $\mathcal{L}=\{G: \chi'_{\mathcal H}(G)=3\}$ is NP-complete, see Holyer \cite{holy}. In fact $\chi'_{\mathcal H}(G)=\chi'(G)$ for all graph $G$, where $\chi'(G)$ is the chromatic index of the graph $G$.
\end{observation}

\begin{theorem}[Cherry Orchard] \label{hard_cherry}
	If ${\mathcal H}=\{P_4, S_4, C_4\}$ then the language $\mathcal{L}=\{G: \chi'_{\mathcal H}(G)=3\}$ is NP-complete.
\end{theorem}

\begin{figure}[htbp]
	\centering
	\includegraphics[scale=0.32]{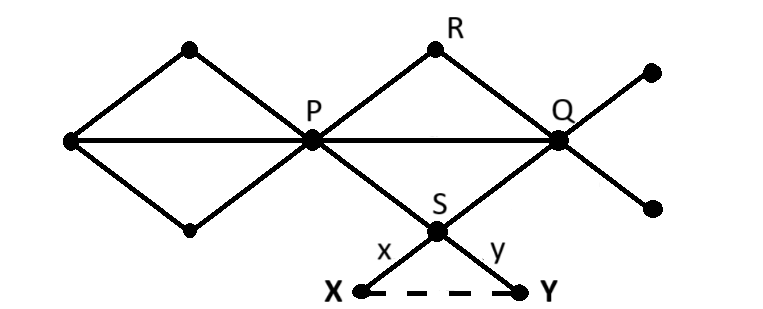}
	\caption {The gadget {\em bird}.}
	\label{bird}
\end{figure}

\noindent {\bf Proof:} Recall that it is NP-complete to decide whether a cubic graph has edge chromatic number 3, see Holyer \cite{holy}. We construct a graph $G^*$ such that $v(G^*)=v(G)+9e(G)$,
$e(G^*)=14e(G)$ and $\chi'_{\mathcal H}(G^*)=3$ iff $\chi'(G)=3$. First let us substitute each edge of $G$ with a gadget, let us call it a {\em bird}, see Figure \ref{bird}. The edges $x$ and $y$ in the figure are called the
{\em legs} of the bird. So, for each $e=(X, Y) \in G$, the edge $e$ is (removed and) replaced with a copy of the bird gadget such that the degree-$1$ end vertices of the legs are glued to the vertices $X$ and $Y$ of $G$ as indicated in Figure~\ref{bird}.

If $\chi'(G)=3$, then the edge colors can be extended to an $\mathcal{H}$-avoiding (cherry-) coloring of $G^*$.
Let the colors be $1, 2, 3$ and assume that an edge $e=(X, Y)$ got the color 1 in the original good coloring. Then in the new (cherry-)coloring both $x$ and $y$ are colored by 1 and let the other edges be colored by 1, 2 or 3, according to Figure \ref{bird3}, which is a good (cherry-)coloring of $G^*$.

\begin{figure}[htbp]
	\centering
	\includegraphics[scale=0.32]{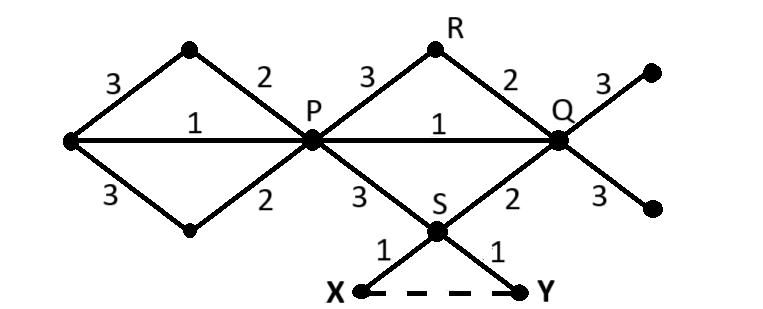}
	\caption {The good (cherry-)coloring of the gadget.}
	\label{bird3}
\end{figure}

To see the other direction, suppose that $G^*$ has an $\mathcal{H}$-avoiding (cherry-)coloring with colors $1, 2, 3$.
If the two legs are monochromatic in each bird gadget, i.e.\ $x$ and $y$ have the same color for each $e$, then by removing the gadgets and undoing the substitution we get a good 3 edge-coloring of $G$.

On the other hand, in a well (cherry-)colored gadget the colors of $x$ and $y$ must be the same. To see this, suppose that $x$ and $y$ have different colors in a gadget. Then, by the pigeonhole principle, among the four edges starting from vertex $S$, there exists
a pair of edges (different from $\{x,y\}$) of the same color.
Similarly, since the degree of $P$ is six, the six incident edges can be grouped into three monochromatic pairs,
which means that $P$ must be the center of three cherries $P_3$. Then the edge $(S,Q)$ must have the same color as $(S,X)$ or $(S,Y)$, because $(P,S)$ is contained in a cherry centered at $P$.
This implies that the five edges incident to $Q$ cannot be covered by cherries using $3$ colors (the edges $(S,Q)$ and $(P,Q)$ are contained in cherries centered at $S$ and $P$, respectively), which is a contradiction.
\bizveg
\\




Thanks to the reviewer's comment, we should note here, that the cherry orchard 2-coloring is also NP-complete, follows from Theorem 2. in \cite{Gonc}. We were not aware of this result previously, but we believe that, regardless of this, it is worth retaining the proof of the 3-coloring case. 
Another accurate observation by the reviewer is that the proof of Theorem~\ref{hard_cherry} can be simplified: in the bird gadget, let P (resp. Q) be adjacent to S and to five vertices of degree 1. Thus, the NP-completeness of cherry orchard 3-coloring holds for 2-degenerate bipartite graphs with girth at least 6. We leave the details to the interested reader.

\section{Further Remarks}

This work can continue in a number of ways. One direction is to consider graphs other than $K_n$.
Real-world graphs motivate the general case, but graphs like $Q_d$, the $d$-dimensional cube, can be interesting from a purely mathematical point of view.
Developing exact or even good approximate algorithms to partition general graphs with special bipartite graphs are also interesting problems and valuable regarding applications. We consider the problem of finding the tight bounds for partitioning $K_n$ with Ferrers graphs (cf.\ Theorem~\ref{NVGtetel}) as the most challenging next step.

\subsection*{Acknowledgement}
We would like to thank the anonymous reviewers. Their comments and suggestions have greatly contributed to the improvement of the manuscript.

\end{document}